\documentclass[11pt,reqno]{amsart}
\usepackage{amsmath,amssymb,amsthm,upref,graphicx,mathrsfs}
\usepackage[toc,page]{appendix}
\usepackage{color,xcolor}
\usepackage[
  colorlinks=true,
  linkcolor=blue,
  citecolor=blue,
  urlcolor=blue]{hyperref}

\numberwithin{equation}{section}

\newtheorem{thm}{Theorem}[section]
\newtheorem{cor}[thm]{Corollary}

\newtheorem{prop}[thm]{Proposition}

\numberwithin{equation}{section}

\let \a=\alpha

\begin{document}

\leftline{ \scriptsize}

\vspace{1.3 cm}
\title
{ On the character tables of symmetric groups}
\author{Kamon Kawsathon $^{a}$ and Kijti Rodtes$^{b,\ast}$ }
\thanks{{\scriptsize
$^{\ast}$ Corresponding author.\\
E-mail addresses: teekawsathon@gmail.com (Kamon Kawsathon ), kijtir@nu.ac.th (Kijti Rodtes).\\
$^{A}$ Department of Mathematics, Faculty of Science, Naresuan University, Phitsanulok 65000, Thailand.\\
$^{B}$ Department of Mathematics, Faculty of Science, Naresuan University, and Research Center for Academic Excellent in Mathematics, Phitsanulok 65000, Thailand.\\}}
\hskip -0.4 true cm

\maketitle


\begin{abstract} In this paper, some zeros and non-zeros in the character tables of symmetric groups are displayed in the partition forms.   In particular, more zeros of self conjugate partitions beside odd permutations are heavily investigated.
\end{abstract}

\vskip 0.2 true cm


\pagestyle{myheadings}
\markboth{\rightline {\scriptsize Kamon Kawsathon and  Kijti Rodtes}}
         {\leftline{\scriptsize }}
\bigskip
\bigskip


\vskip 0.4 true cm

\section{Introduction}

 Representation theory of finite groups has evidently wide application in many areas of mathematics, such as graph theory \cite{N8}, combinatoric theory \cite{GA}, and number theory \cite{N11} et cetera.  Investigating character tables of finite groups is one of the most important and useful topic in this subject.
Even if there are many properties dealing with the construction of character tables of finite groups, there is a very few of such explicit tables.  Also, it seems that there is no a simple way to construct them. However, many researchers devoted times to study some general behavior of character tables, see for example \cite{N16,N12,N14,N13,N10,N15}.

 For symmetric groups $S_n$, there is a question of Navarro to Olsson (2010): $\lq \lq$ \emph{If $p$ is a prime, what are the elements $x$ of the symmetric groups $S_n$ such that $\chi(x) =0$ for all $\chi\in Irr(S_n)$ of degree divisible by $p$?}", \cite{LV}.
In 2015 and 2016, Lucia Morotti  found that the partition of $p$-adic type is $p$-vanishing, but for $p=2,3$ there are some $p$-vanishing conjugacy classes which are not the $p$-adic type, \cite{LV,LP}.  However, for primes $p>3$ such a conjugacy class is not found yet and she made a conjecture that, \lq\lq \emph{for prime $p\geq 5$, all $p$-vanishing classes are $p$-adic types}" (the meanings of $p$-vanishing and $p$-adic types are provided in the next section).

This conjecture motivates us to study conditions on vanishing conjugacy classes for some characters of symmetric groups $S_n$.  Some zeros and non-zeros in the character tables of symmetric groups are displayed as partitions of the forms in Section 3,4.  Some more zeros besides odd permutations of the characters associated to self conjugate partitions are found in Section 5.

\section{Preliminary}
A \textit{partition} of a positive integer $n$ is a tuple $\alpha =(\alpha_{1},\alpha_{2},\dots,\alpha_{r})$ of positive integers $\alpha_{1}\geq \alpha_{2}\geq \dots\geq \alpha_{r}$ such that $\alpha_{1}+\alpha_{2}+ \dots+\alpha_{r}=n$.  The integers $\a_i$'s are called the \textit{parts} of $\a$ and $r:=l(\a)$ the \textit{length} of $\a$ \cite{JB}.  To indicate that $\alpha$ is a partition of $n$, we write $\alpha \vdash n.$  For $i = 1,\dots , n$, if $t_i$ is the number of parts of $\a$ equal to $i$, then we can also write
$\a=(r^{t_r},\dots,1^{t_1})$.  Usually $i^{t_i}$ is left out if $t_i = 0$.
If $\alpha =(\alpha_{1},\alpha_{2},\dots,\alpha_{r})$ is a partition of $n$ then the \textit{Young diagram } $[\alpha]$ of $\alpha$ consists of $n$ boxes placed into $r$ rows, where the $i$-th row has $\alpha_{i}$ boxes.  The box in the $i$-th row and $j$-th column of $[\alpha]$ is called the \emph{$(i,j)$ node} of $[\alpha]$.
For each $i$, denote $\a^{\top}_i$
the number of parts of $\a$ which are bigger than or equal to $i$.  The partition $\a^{\top} =(\a^{\top}_1,\a^{\top}_2,\dots,\a^{\top}_s)$ is called
the \emph{conjugate partition associated with $\a.$}    If $\alpha=\alpha^\top$, then $\alpha$ is called a \textit{self-conjugate partition}.

 If $(i, j)$ is a node of $[\alpha]$ we denote by $H^{\alpha}_{i,j}$  the $(i, j)$-\textit{hook}
of $\alpha$ which is the set of nodes of $[\alpha]$ of the form $(i ,j')$ for some $ j'\geq j$  or $(i', j)$
for some $i'\geq i.$  The \textit{hook-length} $h^{\alpha}_{i,j}$
 of the $(i,j)$-node is equal to the number of nodes in $H^{\alpha}_{i,j}.$  The set of nodes $(l, k)$'s with $l\geq i, k\geq j$ of $\alpha$ such that $(l + 1, k + 1)$ is not in $[\alpha]$ is called the \emph{(i,j)-rim} of $[\alpha]$ and is denoted by $R^{\alpha}_{i,j}$.
  For $h\geq1$, let $w_h(\alpha)$ be
the $h$-\textit{weight }of $\alpha$ which is the maximum number of $h$-hooks which can be
recursively removed from $\alpha$ \cite{GA}.  The $h$-weight of a partition is also equal to the number of its hooks of length divisible by $h$.
  For a partition $\alpha$ of $n$ and $k\in \mathbb{N}$,  we also denote:
$$ I_k^{\alpha}:=\{{(i,j)}\,|\,h^\alpha_{i,j}=k\}.$$

In the symmetric group $S_n$, each conjugacy class of $S_n$ corresponds
naturally to the partitions of $n$ associated to the cycle structure of that class.   The value of the irreducible character $\chi^\alpha$, labeled by the partition $\alpha$, evaluated at the conjugacy class corresponding to a partition $\beta$ can be calculated recursively by the well known \emph{Murnaghan-Nakayama} formula, \cite{GA}.  Precisely,
if $\alpha$ is a partition of $n=k+m$  and $\beta\in S_{n}$ contains a $k$-cycle and $\rho\in S_{m}$ is of cycle type
  deleting $k$-cycle out of $\beta $, then
 $$\chi^{\alpha}(\beta)=\sum_{(i,j )\in I^\alpha_k}(-1)^{l^{\alpha}_{i,j}}\chi^{\alpha\backslash R^{\alpha}_{i,j} }(\rho),$$
where $l^{\alpha}_{i,j}:=\alpha_j^\top-i$ is the \emph{leg length} of the hook $H^{\alpha}_{i,j}$ and $\alpha\backslash R^{\alpha}_{i,j}$ simply denotes the partition associated to the Young diagram $[\alpha]\backslash R^{\alpha}_{i,j}$.  By the \emph{Frame-Robinson-Thrall Hook length} formula,  the degree of $\chi^\alpha$ can be calculated by
$$\chi^{\alpha}(1^n)=\frac{n!}{\prod_{{(i,j)}\in[\a]}h^{\a}_{i,j}}.$$
Note that, for each partition $\alpha$ of $n$, $h^{\a}_{i,j}=h_{j,i}^{\alpha^\top}$, for all $ (i,j)\in[\alpha]$ and hence   $$\chi^{\alpha}(1^n)=\frac{n!}{\prod h^{\a}_{i,j}}=\frac{n!}{\prod h^{\a^{\top}}_{j,i}}=\chi^{\alpha^{\top}}(1^n).$$
Namely, the degree of $\chi^\alpha$ and $\chi^{\alpha^\top}$ are always equal.

Let  $p$ be a prime and $n = a_{0} + a_{1}p + \dots + a_{t}p^{t}$ be the p-adic decomposition of $n$, (with $a_t\neq0$).  A partition of $n$ is of $p$-\textit{adic type} if it is of the form $$(s_{t,1}p^t ,\dots ,s_{t,h_{t}}p^t , \dots , s_{0,1} , \dots , s_{0,h_{0}})$$
 with $(s_{i,1} , \dots , s_{i,h_{i}})\vdash a_i$ for $0\leq i\leq t$.  As $0\leq a_{i}<p$ for $0\leq i\leq k$ we have equivalently that a partition $\alpha=(\alpha_j)_{j\geq0}$ is of
$p$-adic type if and only if $$\sum_{j:p^{i}|\alpha_{j},p^{i+1}\nmid\alpha_{j}}\alpha_j=a_{i}p^{i}$$
for $0\leq i\leq k$, \cite{GM,LV}.
 Let $\chi$ be an irreducible character of
a finite groups and let $p$ be a prime.  We say that $\chi$ is $p$-\textit{singular} if $p$ divides its
degree.  A conjugacy class of a finite group $G$ is
called $p$-\textit{vanishing} if all $p$-singular irreducible characters of $G$ take value $0$
on that conjugacy class.  For irreducible characters of symmetric groups, Lucia Morotti discovered that (Corollary 1.5 in  \cite{LV}):
\begin{thm}\label{2.1}
 Partitions of $p$-adic type are $p$-vanishing.
\end{thm}

\section{Some non-zeros in the character table of $S_n$}
In this section, we provide an observation for conjugacy classes which are not zero under the evaluation of some irreducible characters of finite groups.   For any finite group $G$, element $g\in G$ and an irreducible character $\chi$ of $G$, it is well known that (see, for example, Lemma 2.15 in \cite{N18}) $$\chi(g)=\sum^{\chi(1)}_{i=1}\varepsilon_i,$$ where $\varepsilon_i$ 's are $o(g)$-roots of unity.   For any positive integer $m=p_1^{a_1}\cdots p_r^{a_r}$ (prime factorization), the weight set $W(m)$ means the set of all non-negative integers $k$ in which there are $m$-roots of unity $\varepsilon_1,\dots,\varepsilon_k$ such that $\varepsilon_1+\cdots+\varepsilon_k=0$.  The main theorem of  T.Y.  Lam and K.H.  Leung in \cite{N17} asserts that the weight set $W(m)$ is exactly given by $\mathbb{N}_{0}p_1+\cdots +\mathbb{N}_{0}p_r$, where $\mathbb{N}_{0}$ denotes the set of all non-negative integers.  Therefore the following is immediate:

\begin{thm}\label{3.1}
 Let $\alpha$ and $\beta=(\beta_1,\dots,\beta_k) $ be partitions of $n$, $m=\operatorname{lcm}(\beta_1,\dots,\beta_k)=p_1^{a_1}\cdots p_r^{a_r}$ and $W(m)=\{n_1p_1+\cdots+n_rp_r\, |\,n_i\in\mathbb{N}_0\}.$  If $\chi^{\a}(1^n)\not\in W(m),$ then $\chi^{\a}(\beta)\neq 0$.
\end{thm}
In particular,
\begin{cor}\label{3.2}
Let $\a$ and $\beta=(\beta_1,\dots,\beta_k)$ be  partitions of $n$ such that $\operatorname{lcm}(\beta_1,\dots,\beta_k)=p^t$ for some integer $t>0$.  If $p\nmid \chi^{\a}(1^n)$, then $\chi^{\a}(\beta)\neq 0.$
\end{cor}
For example, if $\a=(2,1^{n-2})$ and $\beta$ be a partition of $n$ such that $\operatorname{lcm}(\beta_1,\dots,\beta_k)=p^t$ for some integer $t>0$ and $p\nmid  (n-1)$, then $\chi^{\a}(\beta)\neq 0$ (because $\operatorname{deg}(\chi^\a)=n-1$).

Note from the Diophantine Frobenius problem that the largest number that cannot be written in the form $$ \sum_{i=1}^n a_i x_i, \quad x_i\in \mathbb{N}_0,$$ for given positive integers $a_1,\dots,a_n$ with $\operatorname{gcd}(a_1,\dots,a_n)=1$ is called the \emph{Froenius number} and denoted by $g(a_1,\dots,a_n)$.  It is well known that for $a_1,a_2\in \mathbb{N}$ with $\operatorname{gcd}(a_1,a_2)=1$, then $g(a_1,a_2)=a_1a_2-a_1-a_2$.  So the following is an immediate consequence of Theorem \ref{3.1}.

\begin{cor}\label{3.4}
Let $\a$ and $\beta=(\beta_1,\beta_2,\dots,\beta_k)$ be  partitions of $n$ such that $\operatorname{lcm}(\beta_1,\beta_2,\dots,\beta_k)$ is $p^{a_1}q^{a_2}$  (prime factorization) for some integer $a_1,a_2>0$.  If $ \chi^{\a}(1^n)=pq-p-q$, then $\chi^{\a}(\beta)\neq 0$.
\end{cor}
For example, if $\a=(2,1^{n-2})$ and $\beta$ are partitions of $n=(p-1)(q-1)$ such that $\operatorname{lcm}(\beta_1,\beta_2,\dots,\beta_k)=p^t q^s$ for some positive integer $t,s$, then   $\chi^{\a}(1^n)=pq-p-q$. Hence $\chi^{\a}(\beta)\neq 0$.

\section{Some zeros in the character table of $S_n$}
 First, we observe that:
\begin{prop}\label{4.1}
Let $\alpha$ and $\beta$ be partitions of $n$ such that $\alpha=(\alpha_1,1^{n-\alpha_1})$ with $n-\alpha_1\geq\alpha_1>1 .$  If $\beta=(\beta_1,\dots,\beta_r,\alpha_1,1)$, then $\chi^\alpha(\beta)=0$.
\end{prop}
\begin{proof} Let $k\in\{1,\dots,r\}$.   Denote $\a^{(k)}$ the partition associated to the Young diagram obtained by removing the hook of length $\beta_{k}$ from $[\alpha^{(k-1)}]$, where $[\a^{(0)}]=[\a]$. 
 Since $\alpha=(\alpha_1,1^{n-\alpha_1})$ and $n-\alpha_1\geq \alpha_1$, $[\a]$ contains exactly one hook of length $l$ for each $n-\a_1\geq l\geq\alpha_1$.         Since $\beta_k\geq \alpha_1$, we have that $I^{\alpha^{(k-1)}}_{\beta_k}$ contains at most one element.  If $I^{\alpha^{(k-1)}}_{\beta_k}$ has no element, $\chi^\alpha(\beta)=0$, by Murnaghan-Nakayama formula.  Suppose that $I^{\alpha^{(k-1)}}_{\beta_k}$ has exactly one element and let $ I^{\alpha^{(k-1)}}_{\beta_k}=\{(i_k,j_k)\}$.  Then, by Murnaghan-Nakayama formula,
$$\chi^\alpha(\beta)= (-1)^{l^{\a}_{(i_1,j_1)}}(-1)^{l^{\alpha^{(1)}}_{(i_2,j_2)}}\cdots(-1)^{l^{\alpha^{(r-1)}}_{(i_{r},j_{r})}}\chi^\gamma(\alpha_1,1).$$
Let $\gamma$ be the partition associated to the Young diagram obtained by removing a sequence of hooks of lengths $\beta_1,\beta_2,\dots,\beta_r$ from $\a$; i.e., $\gamma=(\alpha_1,1)$.  Since $[\gamma]$ does not contain a hook of length $\a_1 $, we now conclude that   $\chi^{\alpha}(\beta)=0$.
\end{proof}

Let $\alpha$ be a  partition of $n$.  If $\gamma$ is the partition associated to the Young diagram $[\gamma]$ obtained by the process $P^\alpha(\beta_1,\beta_2,\dots,\beta_s)$:
 \begin{center}
 \lq\lq \emph{removing a hook of length $\beta_1$  out of $[\alpha]$  at node $(i_1,j_1)$ following by removing a hook of length $\beta_2$ out of $[\alpha]\setminus R^\alpha_{(i_1,j_1)}$ at node $(i_2,j_2)$ of $[\alpha]\setminus R^\alpha_{(i_1,j_1)}$  and so on till the s-th step}",
\end{center}
then we will denote $\gamma$ by $\alpha^{(s)}_{\vec{a}}(\beta_1,\beta_2,\dots,\beta_s)$, where $\vec{a}=(a_1,a_2,\dots,a_s)$ is the finite sequences of pairs of positive integers $a_1=(i_1,j_1),a_2=(i_2,j_2),\dots, a_s=(i_s,j_s)$.  Denote $I^\alpha(\beta_1,\beta_2,\dots\beta_s)$ the set of all sequences $\vec{a}$ of pair of the positive integers  for which the process $P^\alpha( \beta_1,\beta_2,\dots,\beta_s)$ can be done.  Note that if $s=1$, then $I^\alpha(\beta_1)=I^\alpha_{\beta_1}$.

\begin{thm}\label{4.2}
 Let $\alpha$ and $\beta=(\beta_1,\dots,\beta_s,\dots,\beta_k) $ be partitions of $n$, for some $1\leq s< k$.  Let $p$ be a prime.  If $p|\operatorname{deg}(\chi^{\alpha^{(s)}_{\vec{a}}(\beta_1,\dots,\beta_s)})$ for all $\vec{a}\in I^\alpha( \beta_1,\dots,\beta_s) $ and $(\beta_{s+1},\dots,\beta_k)\vdash m$, $(1<m)$ is a partition of $p$-vanishing, then $\chi^\alpha(\beta)=0$.
\end{thm}
\begin{proof} We first remove $\beta_1$ out of $[\alpha]$ at all possible nodes.  By Murnaghan-Nagayama formula, we have
 $$\chi^{\alpha}(\beta)=\sum_{(a_1) \in I^{\a}(\beta_1) }(-1)^{l^{\a}_{a_1}}\chi^{\alpha_{(a_1)}^{(1)}(\beta_1)}(\beta_2,\dots,\beta_s,\dots,\beta_k).$$
Next, we remove $\beta_2$ out of $[\a^{(1)}_{(a_1)}(\beta_1)]$ at all possible nodes.  We then have
$$\chi^{\alpha}(\beta)=\sum_{(a_1) \in I^{\a}(\beta_1) }(-1)^{l^{\a}_{a_1}}(\sum_{a_2 \in I^{\a^{(1)}_{(a_1)}(\beta_1)}_{\beta_2} }(-1)^{l^{\a^{(1)}_{(a_1)}(\beta_1)}_{a_2}} \chi^{\alpha^{(2)}_{(a_1,a_2)}(\beta_1,\beta_2)}(\beta_3,\dots,\beta_s,\dots,\beta_k)).$$
By repeating this process, we conclude that
\begin{eqnarray*}
   \chi^{\alpha}(\beta)&=&\sum_{(a_1) \in I^{\a}(\beta_1) }(-1)^{l^{\a}_{a_1}} \cdots  \sum_{a_s \in I^{\alpha^{(s-1)}}_{\beta_s} }(-1)^{l^{\alpha^{(s-1)}}_{a_s}}    \chi^{\alpha^{(s)}_{\vec{a}}(\beta_1,\dots,\beta_s)}(\beta_{s+1},\dots,\beta_k)  \\
   &=&   \sum_{\vec{a}\in I^\alpha(\beta_1,\dots,\beta_s) } \pm\chi^{\alpha^{(s)}_{\vec{a}}(\beta_1,\dots,\beta_s)}(\beta_{s+1},\dots,\beta_k)
\end{eqnarray*}
where $\vec{a}=(a_1,\dots,a_s)$, $I^{\alpha^{(s-1)}}_{\beta_s}=I^{\alpha^{(s-1)}_{(a_1,\dots,a_{s-1})}(\beta_1,\dots,\beta_{s-1})}_{\beta_s} $ and $l^{\alpha^{(s-1)}}_{a_s}=l^{\alpha^{(s-1)}_{(a_1,\dots,a_{s-1})}(\beta_1,\dots,\beta_{s-1})}_{a_s} $.
Since $(\beta_{s+1},\dots,\beta_k)\vdash m$ is a partition of $p$-vanishing and $p|\operatorname{deg}(\chi^{\alpha^{(s)}_{\vec{a}}(\beta_1,\dots,\beta_s)})$ for all $\vec{a}\in I^\alpha( \beta_1,\dots,\beta_s) $,  by Theorem \ref{2.1},
$\chi^{\alpha^{(s)}_{\vec{a}}(\beta_1,\dots,\beta_s)}(\beta_{s+1},\dots,\beta_k)=0$. Therefore $\chi^\alpha(\beta)=0$.
\end{proof}
For $s=1$ in Theorem \ref{4.2}, we have in particular that:
\begin{cor}\label{4.3} Let $p$ be a prime, $a,c,l$ be positive integers such that $a\geq cp+1$ and $k$ be a non-negative integer such that $a+cp+2l+k+1=n$.
Let $\alpha,\beta\vdash n$ with $\a=(a,cp+1,2^l,1^k)$ and $\beta=(a+l+k+1,\gamma)$ with $\gamma\vdash cp+l=:m$.  If $p\nmid m$ and $\gamma$ is of $p$-vanishing then $\chi^\alpha(\beta)=0$.
\end{cor}
\begin{proof} Since $h^\a_{1,1}=a+l+k+1$, we have that $$\a^{(1)}:=\a^{(1)}_{((1,1))}(h^\a_{1,1})=(cp,1^{l})$$ which is a partition of $m=n-h^\a_{1,1}=cp+l$ associated to the Young diagram obtained by removing the hook of lengths $h^\a_{1,1}$ out of $[\a]$.  The Frame-Robinson-Thrall Hook length formula implies that
$$\deg(\chi^{\a^{(1)}})=\chi^{\a^{(1)}}(1^m)=\frac{(cp+l)!}{(cp+l)(cp-1)!l!}=\frac{(cp-1+l)!}{(cp-1)!l!}=\left(
                                                                                                           \begin{array}{c}
                                                                                                             (cp-1)+l \\
                                                                                                             l \\
                                                                                                           \end{array}
                                                                                                         \right).
  $$
We now write $cp-1=e_sp^s+e_{s-1}p^{s-1}+\cdots+e_1p+e_0$ and $l=d_rp^r+d_{r-1}p^{r-1}+\cdots+d_1p+d_0$ as $p$-adic decompostitions.  Since $p\nmid m$, $p\nmid l$; i.e., $1\leq d_0 <p$.  Also, $e_0=p-1$.  Thus, the number of carries when $l$ is added to $cp-1$ in base $p$ is at least $1$.  By Kummer's Theorem,  $$\nu_p(\left(
                                                                                                           \begin{array}{c}
                                                                                                             (cp-1)+l \\
                                                                                                             l \\
                                                                                                           \end{array}
                                                                                                         \right))\geq 1,$$ Namely,  $p\mid \deg(\chi^{\a^{(1)}})$.  Hence, the result follows by Theorem \ref{4.2}.
\end{proof}

Moreover, for $s=2$ in Theorem \ref{4.2}, we also have in particular that:
\begin{cor}\label{4.4} Let $p$ be a prime, $a,b,c,l,t$ be positive integers such that $a\geq b\geq cp+2$ and $k$ be a non-negative integer such that $a+b+cp+3l+2t+k+2=n$.
Let $\alpha,\beta\vdash n$ with $\a=(a,b,cp+2,3^l,2^t,1^k)$ and $\beta=(\beta_1,\beta_2,\gamma)$ with $\gamma\vdash cp+l=:m$, where $\beta_1=a+l+t+k+2$ and $\beta_2=b+l+t$.   If $p\nmid m$ and $\gamma$ is of $p$-vanishing then $\chi^\alpha(\beta)=0$.
\end{cor}
\begin{proof} Since $h^\a_{1,1}=a+l+t+k+2=\beta_1$, we have that $$\a^{(1)}:=\a^{(1)}_{((1,1))}(h^\a_{1,1})=(b-1,cp+1,2^{l},1^t)$$ which is a partition associated to the Young diagram obtained by removing the hook of lengths $\beta_1$ out of $[\a]$. Since
$\beta_2=b+l+t=h^{\a^{(1)}}_{1,1}$,  we have
 $$ \a^{(2)}:=\a^{(2)}_{((1,1),(1,1))}(\beta_1,\beta_2)=(cp,1^{l})\vdash m=n-\beta_1-\beta_2=cp+l$$ is a partition associated to the Young diagram obtained from  $[\alpha^{(1)}]$ by removing hooks of lengths $\beta_2$.   The same arguments as in the proof of Corollary \ref{4.3} can be used to complete the proof.
\end{proof}

\section{Zeros of self conjugate partitions}
Recall from Lemma 2.1.8 in \cite{GA}  that $\chi^{\a^\top}(\beta)=\chi^\a(\beta)$ if $\beta$ is even and $\chi^{\a^\top}(\beta)=-\chi^\a(\beta)$ if $\beta$ is odd.  Hence, if $\a^\top=\alpha$, then $\chi^\a(\beta)=0$ for any odd $\beta$.  However, there are some more zeros for self conjugate partitions.

\begin{prop}\label{4.12}\label{finalt}
 Let $\alpha=\alpha^\top$ and $\beta$ be partitions of $n$.  If $\beta_1$ is an even part of $\beta$ such that $\beta_1>\frac{n}{2}$ and $\gamma\vdash (n-\beta_1)$, then $\chi^\alpha(\beta_1,\gamma)=0.$
\end{prop}
\begin{proof}
 Since $\alpha=\alpha^\top$,  $h^\alpha_{1,1}>h_{1,2}^\alpha=h_{2,1}^\alpha>h_{i,j}^\alpha$ for all $(i,j)\neq(1,1),(1,2),(2,1)$ and $h_{1,1}^\a$ is odd.
Thus, $w_{h_{1,2}^\alpha}(\alpha)=2$.  So $2h_{1,2}^\alpha\leq n$ and then $h_{1,2}^\alpha\leq\frac{n}{2}$.  Since $\beta_1$ is even and $h_{1,1}^\a$ is odd, $\beta_1\neq h^{\a}_{1,1}$.  Moreover, by the assumption that $\beta_1>\frac{n}{2}$, we have that  the Young diagram of $[\alpha]$ does not contain a hook of length $\beta_1$.  Hence $\chi^\alpha(\beta_1,\gamma)=0$.
\end{proof}
According to Proposition \ref{finalt}, if $\alpha$ is a self conjugate partition and $\beta$ is even, it does not necessary to have that $\chi^\a(\beta)\neq 0$; for example, if $\alpha=(13,5,2^3,1^8)$ which is self conjugate and $\beta=(20, 5, 2^3,1,)$ which is even, then $\chi^\a(\beta)=0$. \\

In the remaining, we concentrate only on self conjugate partitions $\alpha$ of $n$.   Let $\a=(r_1^{k_1},\dots,r_m^{k_m})$ with $r_1>\cdots> r_m$ and $k_i>0$ for $1\leq i\leq m$. Since $\a$ is self conjugate, we have that $r_i=\sum_{j=1}^{m-i+1}k_j$ for $1\leq i\leq m$ and then also that $k_i=r_{m-i+1}-r_{m-i+2}$ for $1\leq i\leq m$.   Let $s=\lceil m/2\rceil$, the ceiling function of $m/2$.  We also denote $\lfloor q \rfloor$ the floor function of the real number $q$.

For each $1\leq i,j\leq m$, we denote $A_{i,j}$ the $k_i\times k_j$ matrix whose its entries are hook lengths of $[\a]$ in the strip $(i,j)$ which is the set of nodes
$$\{(x,y)\in \mathbb{N}\times \mathbb{N} \,|\, k_1+\cdots +k_{i-1}<x\leq k_1+\cdots +k_{i} \hbox{ and } k_1+\cdots +k_{j-1}<y\leq k_1+\cdots +k_{j}\}  $$ in $[\a]$.  Namely, we can consider the set of all hook lengths of $[\a]$ as a block matrix $A$ in the form
$$A=\begin{bmatrix}
A_{1,1} & A_{1,2}&\dots&A_{1,m} \\
A_{2,1} & A_{2,2}&\dots&A_{2,m} \\
\vdots & \vdots&\ddots&\vdots \\
A_{m,1} & A_{m,2}&\dots&A_{m,m} \\
\end{bmatrix},
$$
with $A_{i,j}=0$ if $i+j\geq m+2$, in particular if $i,j\geq s+1$.   This matrix is symmetric because $\alpha$ is self conjugate.  Thus, it suffices to investigate the matrices $A_{i,j}$ with $1\leq i \leq s$ and $i\leq j \leq m$.  In the following, for positive integers $n_1<n_2$, we denote $[n_1,n_2]:=\{k\in \mathbb{Z} \,|\, n_1\leq k \leq n_2 \}$ and also define $[n_1,n_2]=\emptyset$ if $n_1>n_2$.  We also denote $e(A_{i,j})$ the set of all entries in $A_{i,j}$.  By direct computation, we have that, for each $1\leq i \leq s$ and $i\leq j \leq m$, $$e(A_{i,j})=[r_i+r_j-(\sum^{i}_{t=1}k_t)-(\sum^{j}_{t=1}k_t)+1,r_i+r_j-(\sum^{i-1}_{t=1}k_t)-(\sum^{j-1}_{t=1}k_t)-1].$$

Moreover, for each $1\leq i \leq \lfloor m/2 \rfloor$ and $1\leq j \leq m-1$ with $i\leq j$ and $i+j \leq m$, we denote
\begin{equation}\label{minmaxG1}
    G_{i+1,j+1}=[\max \big(e(A_{i+1,j+1})\big)+1,\min \big(e(A_{i,j})\big)-1]
\end{equation}
 and, for each $1\leq j\leq m$,
  \begin{equation}\label{minmaxG2}
    G_{1,j}=\big[{r_1+r_j-(\sum_{t=1}^{j-1}k_t)},n\big].
 \end{equation}
Note by the direct calculation that $G_{a,b}=\emptyset$ for $a+b\geq m+3$.  Then
\begin{equation}\label{eqGD}
    G_{i,i+j-1}=\bigg[r_{i}+r_{i+j-1}-(\sum^{i-1}_{t=1}k_t)-(\sum^{i+j-2}_{t=1}k_t),r_{i-1}+r_{i+j-2}-(\sum^{i-1}_{t=1}k_t)-(\sum^{i+j-2}_{t=1}k_t)\bigg],
\end{equation}
for all $1\leq i\leq  \lfloor  (m-j+3)/2\rfloor=:M_j$.
In the following results,  for each $1\leq j\leq m$, define
\begin{equation}\label{defgj}
    G_{j}:= \bigcup_{i=1}^{M_j} G_{i,i+j-1}.
\end{equation}
Note from (\ref{eqGD}) that the union in (\ref{defgj})  is a disjoint union; namely, $G_{i,i+j-1}\cap G_{k,k+j-1}=\emptyset$ if $i\neq k$.  Moreover, for each $j$, we have that $\min(G_{i,i+j-1})>\max(G_{i+1,i+j})$ for each $i=1,\dots,M_j-1$ and thus the smallest element of $G_j$ belongs to $G_{M_j,M_j+j-1}$.  Let $H^\a$ be the set of all hook lengths of $[\a]$.
\begin{prop}\label{5.13}
For a self conjugate partition $\a$ of $n$ and $1\leq x\leq n$ integer, we have that  $x \notin H^{\a} $ if and only if $x\in G(\alpha)$, where $G(\alpha):= G_{1}\cap G_{2}\cap\cdots\cap G_{m} $.
\end{prop}
\begin{proof}
 Suppose that $x\in G(\a)$.
  For each $j\in {1,2,\dots,m}$, there exists $i\in\{1,2,\dots,M_j\}$ such that $x\in G_{i,i+j-1}$.
By (\ref{minmaxG1}) and (\ref{minmaxG2}), we see that $x>\max(e(A_{i,i+j-1}))$ and $x<\min(e(A_{i-1,i+j-2}))$. This implies that $ x\not \in  e(A_{i+t,i+j-1+t})$ for each $0\leq t\leq M_j-i$ and $x\not\in e(A_{i-t,i+j-1-t})$ for each $1\leq t\leq i-1$.  Namely, $x$ does not appear in the $j$-diagonal (South-East) strip $$\bigcup^{M_j}_{i=1} e(A_{i,i+j-1})$$ of $A$.  Since $j$ is arbitrary, $x$ is not an entry of $A$ which means that $x \notin H^{\a} $.

   On the other hand,  suppose that  $x\not\in G(\a) $.  Then there exists $j\in\{1,2,\dots,m\}$ such that $x\not\in G_j$.  Note that, for each $j=1,\dots,m$, $$G_j\cup \bigcup^{M_j}_{i=1} e(A_{i,i+j-1})=\{1,\dots,n\}.$$  Then $x\in e(A_{i,i+j-1})$ for some $1\leq i\leq M_j$, and hence $x\in H^{\a}$.
\end{proof}

By the distributive law of sets and the definition of $G_j$'s above, we have that
  \begin{equation}\label{intersectionofG}
    G(\a)= \bigcup (G_{i_1,1+i_1-1}\cap G_{i_2,2+i_2-1}\cap\cdots \cap G_{i_{m},m+i_m-1}),
\end{equation}
where the union runs over the set $$I:=\{(i_1,i_2,\dots,i_m) \in \mathbb{Z}^m \,|\,  1\leq i_j\leq M_j \hbox{ for each } 1\leq j\leq m\}.$$
 There are exactly
$$\prod_{j=1} ^m M_j$$
terms in the union form of $G(\alpha)$ in (\ref{intersectionofG}).   Since $G_{1,1}\cap G_{1,2}\cap \cdots\cap G_{1,m}=[2r_1,n]$, there are at most $2r_1-1$ non-empty terms in (\ref{intersectionofG}).
\begin{prop}\label{emptynode} Let $Y=G_{i_1,1+i_1-1}\cap G_{i_2,2+i_2-1}\cap\cdots \cap G_{i_{m},m+i_m-1}$ be a term in the union form of $G(\a)$ in (\ref{intersectionofG}).
\begin{enumerate}
  \item If there exists $1\leq l \leq m-1$ such that $i_l-i_{l+1}\geq 2$ or $i_{l+1}-i_l\geq 1$, then $Y=\emptyset$.
   \item If there exist $1\leq l \leq m-2$ such that $i_l=i_{l+1}=i_{l+2}$ and $r_{i_l}+r_{i_l+l-1}\geq r_{i_l-1}+r_{i_l+l}-k_{i_l+l}-k_{i_l+l-1}$, then $Y=\emptyset$.
  \item If there exist $1\leq l \leq m-2$ such that $i_{l}=i_{l+1}+1=i_{l+2}+2$ and $r_{i_{l+2}}+r_{i_{l+2}+l+1}\geq r_{i_{l+1}}+r_{i_{l+2}+l}-k_{i_{l+1}}-k_{i_{l+2}}$, then $Y=\emptyset$.
\end{enumerate}
\end{prop}
\begin{proof}  Suppose that there exits $1\leq l \leq m-1$ such that $i_l-i_{l+1}\geq 2$ or $i_{l+1}-i_l\geq 1$.  If $i_l-i_{l+1}\geq 2$, then $i_l-1>i_{l+1}$ and $i_{l+1}+l\leq i_l+l-2$.  Thus $r_{i_l-1}< r_{i_{l+1}}$ and $r_{i_l+l-2}\leq r_{i_{l+1}+l}$.   By (\ref{eqGD}), it is now straightforward to conclude that    $$\max(G_{i_l,i_l+l-1})<\min(G_{i_{l+1},i_{l+1}+l}) $$
 which means that $G_{i_l,i_l+l-1}\cap G_{i_{l+1},i_{l+1}+l}=\emptyset$ and hence $Y=\emptyset$.  If $i_{l+1}-i_l\geq 1$, then $i_{l+1}>i_l$ and $i_{l+1}-1\geq i_l$.  Thus $r_{i_{l+1}-1}\leq r_{i_l}$ and $r_{i_{l+1}+l-1}<r_{i_l+l-1}$.  By (\ref{eqGD}), it is now straightforward to conclude that   $$\min(G_{i_l,i_l+l-1})>\max(G_{i_{l+1},i_{l+1}+l}) $$
 which means that  $G_{i_1,i_1+1-1}\cap G_{i_2,i_2+2-1}=\emptyset$ and hence $Y=\emptyset$.

Suppose that there exist $1\leq l \leq m-2$ such that $i_l=i_{l+1}=i_{l+2}$.  By (\ref{eqGD}), we have that $$\min(G_{i_l,i_l+l-1})>\min(G_{i_{l+1},i_{l+1}+l})>\min(G_{i_{l+2},i_{l+2}+l+1}),$$
and
$$\max(G_{i_l,i_l+l-1})>\max(G_{i_{l+1},i_{l+1}+l})>\max(G_{i_{l+2},i_{l+2}+l+1}).$$
So, $$G_{i_l,i_l+l-1}\cap G_{i_{l+1},i_{l+1}+l} \cap G_{i_{l+2},i_{l+2}+l+1} =[\min(G_{i_l,i_l+l-1}), \max(G_{i_{l+2},i_{l+2}+l+1})].$$
This set is non-empty when $r_{i_l}+r_{i_l+l-1}\leq r_{i_l-1}+r_{i_l+l}-k_{i_l+l}-k_{i_l+l-1}$.  Similar arguments can be applied to conclude the remaining.

\end{proof}

According to Proposition \ref{emptynode}, a possibly nonempty set $$Y=\bigcap^m _{k=1}G_{i_k,k+i_k-1}$$
in the union form of $G(\a)$ in  (\ref{eqGD})  must satisfy the conditions $i_k-i_{k+1}\in \{ 0,1\}$ for each $k=1,\dots,m$.  In other words, the non-empty set  $Y$  is the intersection of nodes, $G_{i_k,k+i_k-1}$'s in a North-East ladder (possibly with different steps) of $[\a]$.

For each $1\leq v \leq \lfloor \frac{m}{2} \rfloor+1$, let $L^N_v$ be the one step $v$th North-East ladder starting in the North direction of $[\a]$ defined by $$L^N_{v}:=G_{v,v}\cap G_{v-1,v}\cap G_{v-1,v+1}\cap G_{v-2,v+1}\cap \cdots \cap G_{2,2v-2}\cap G_{1,2v-2}\cap G_{1,2v-1}\cdots\cap G_{1,m}. $$
Note that $G_{1,i}\cap G_{1,k}=G_{1,i}$ if $i<k$.  Thus,
\begin{eqnarray*}
L^N_{v}&=&\big(\bigcap_{ 0\leq i< v-1}(G_{v-i,v+i}\cap G_{(v-1)-i,v+i})\big)\cap\big(\bigcap_{i\geq 2v-1}G_{1,i}\big)\\
&=&\bigcap_{ 0\leq i< v-1}(G_{v-i,v+i}\cap G_{(v-1)-i,v+i})
\end{eqnarray*} for all $ 2\leq v\leq \lfloor \frac{m}{2} \rfloor+1$, and $L^N_1={G_{1,1}}$.

Similarly, for each $1\leq v \leq \lfloor \frac{m+1}{2} \rfloor$, let $L^E_v$ be the one step $v$th North-East ladder starting in the East direction of $[\a]$ defined by $$L^E_{v}:=G_{v,v}\cap G_{v,v+1}\cap G_{v-1,v+1}\cap G_{v-1,v+2}\cap \cdots \cap G_{2,2v-1}\cap G_{1,2v-1}\cap G_{1,2v}\cdots\cap G_{1,m}. $$
Using the same arguments as above we have that
$$ L^E_v=\bigcap_{ 0\leq i\leq v-1}(G_{v-i,v+i}\cap G_{v-i,v+1+i}), $$
for all $ 2\leq v\leq \lfloor \frac{m+1}{2} \rfloor$, and $L^E_1={G_{1,1}}$. \\

For the following results, for each $2\leq v \leq \lfloor \frac{m}{2} \rfloor+1$, we denote
 \begin{eqnarray*}
a^N_v&=&\max \big\{r_{v-i}+r_{v+i}-(\sum^{v-i-1}_{t=1}k_t)-(\sum^{v+i-1}_{t=1}k_t)\,\mid\, 0\leq i<v-1\big\} , \\
b^N_v&=&\min \big\{r_{v-i-1}+r_{v+i-1}-(\sum^{v-i-1}_{t=1}k_t)-(\sum^{v+i-1}_{t=1}k_t)\,\mid\, 0\leq i<v-1\big\} , \\
c^N_v&=&\max\big\{r_{v-i-1}+r_{v+i}-(\sum^{v-i-2}_{t=1}k_t)-(\sum^{v+i-1}_{t=1}k_t)\,\mid\, 0\leq i<v-1\big\},  \\
d^N_v&=&\min\big\{r_{v-i-2}+r_{v+i-1}-(\sum^{v-i-2}_{t=1}k_t)-(\sum^{v+i-1}_{t=1}k_t)\,\mid\, 0\leq i<v-1\big\},  \\
 \end{eqnarray*}
and for each $2\leq v \leq \lfloor \frac{m+1}{2} \rfloor$, we denote
\begin{eqnarray*}
a^E_v&=&\max \big\{r_{v-i}+r_{v+i}-(\sum^{v-i-1}_{t=1}k_t)-(\sum^{v+i-1}_{t=1}k_t)\,\mid\, 0\leq i \leq v-1\big\} , \\
b^E_v&=&\min \big\{r_{v-i-1}+r_{v+i-1}-(\sum^{v-i-1}_{t=1}k_t)-(\sum^{v+i-1}_{t=1}k_t)\,\mid\, 0\leq i \leq v-1\big\} , \\
c^E_v&=&\max\big\{r_{v-i}+r_{v+i+1}-(\sum^{v-i-1}_{t=1}k_t)-(\sum^{v+i}_{t=1}k_t)\,\mid\, 0\leq i \leq v-1\big\},  \\
d^E_v&=&\min\big\{r_{v-i-1}+r_{v+i}-(\sum^{v-i-1}_{t=1}k_t)-(\sum^{v+i}_{t=1}k_t)\,\mid\, 0\leq i \leq v-1\big\}.  \\
 \end{eqnarray*}
Here, we set any sum containing $r_0$ or $k_{m+1}$ to be $n$ and set $r_{m+1}=0$.

\begin{thm}\label{mainscp}  Let $\a=({r_{1}^{k_1}},{r_{2}^{k_2}},\dots,{r_{s}^{k_s}},{r_{s+1}^{k_{s+1}}}\dots,{r_{m}^{k_{m}}})$ be a self conjugate partition.  If $\beta$ is a partition of $n$ and $\beta$ has a part $$x\in[\max\{ a^N_v,c^N_v\},\min\{b^N_v,d^N_v \}]\cup [\max\{ a^E_v,c^E_v\},\min\{b^E_v,d^E_v \}],$$ for some $ v\geq 2$, then $\chi^\a(\beta)=0$.
\end{thm}
\begin{proof} Let $\varepsilon \in \{ N,E\}$.  Assume the assumption and consider $L^\varepsilon_v$ as $T^\varepsilon_v\cap B^\varepsilon_v$, where  $$T^N_v=\bigcap_{ 0\leq i< v-1}G_{v-i,v+i},\quad B_v^N= \bigcap_{ 0\leq i< v-1} G_{(v-1)-i,v+i},$$  and $$T^E_v=\bigcap_{ 0\leq i\leq v-1}G_{v-i,v+i}, \quad B_v^E=\bigcap_{ 0\leq i\leq v-1}G_{v-i,v+1+i}.$$
It is a direct computation from (\ref{eqGD}) that $a^\varepsilon=\min(T^\varepsilon_v)$,  $b^\varepsilon=\max(T^\varepsilon_v)$,  $c^\varepsilon=\min(B^\varepsilon_v)$ and $d^\varepsilon=\max(B^\varepsilon_v)$.  Then $L^\varepsilon_ v= [\max\{ a^\varepsilon_v,c^\varepsilon_v\},\min\{b^\varepsilon_v,d^\varepsilon_v \}]$.   Note that $L^\varepsilon_v\subseteq G(\a)$ and then, by Proposition \ref{5.13}, $x\notin H^\alpha$ for any $x\in L^\varepsilon_v$.  By Murnaghan-Nakayama formula, we complete the proof.

 \end{proof}
 Note that $L^\varepsilon_1=G_{1,1}=[2r_1,n]$.  Thus, if $\beta$ contains a part $x\geq 2r_1$, then $\chi^\a(\beta)=0$, by Murnaghan-Nakayama formula.  For $v=2$ and $m\geq 2$, we compute that $a^N_2=2(r_2-k_1), \, b^N_2=2(r_1-k_1), \, c^N_2=r_1+r_2-k_1$ and $d^N_2=n$.   Therefore, the following is immediate.
 \begin{cor} Let $\a=({r_{1}^{k_1}},{r_{2}^{k_2}},\dots,{r_{s}^{k_s}},{r_{s+1}^{k_{s+1}}}\dots,{r_{m}^{k_{m}}})$ be a self conjugate partition of $n$ with $m\geq 2$.  Then $\chi^\a(\beta)=0$ for all partition $\beta$ of $n$ containing a part $x\in [\max\{2(r_2-k_1),r_1+r_2-k_1 \},2(r_1-k_1)]$.
 \end{cor}
   
   Moreover, the smaller $x$ belongs to $G(\a)$, the larger number of zero occurs in the row $\chi^\a$ (for any self conjugate partition $\a$).  The smallest element of $G(\a)$ belongs to $L^N_{s+1}$ for $m=2s$ or belongs to $L^E_{s}$ for $m=2s-1$, if they are not empty.
   
   The following corollaries are illustration of the usage of Theorem \ref{mainscp}.  The partition forms in the corollaries are found under the assumption that  $L^N_{s+1}$ for $m=2s$ or $L^E_{s}$ for $m=2s-1$ is not empty in some specific conditions.  However, after the form of $\a$ is explicit, there are several different ways to conclude the corollaries by Murnaghan-Nakayama formula as well.

 \begin{cor} Let $s,x,y$ be positive integers with $x\leq y$ and $s\geq 2$.  Let $$\a=((sx+sy)^x, (sx+(s-1)y)^x,\dots, (sx+y)^x,(sx)^y,((s-1)x)^y,\dots,x^y),$$
 and $\beta$ be partitions of $n=sx(s(x+y)+y)$.  If $\beta$ contains a part $x+y$ or $2(x+y)$, then $\chi^\a(\beta)=0$.
 \end{cor}
 \begin{proof} The given partition $\a$ is a self conjugate partition.  We compute that $a^N_{s+1}=0$,  $b^N_{s+1}=x+y=c^N_{s+1}$ and 
  $$d^N_{s+1}=\left\{
          \begin{array}{ll}
            2(x+y), & \hbox{ if $s\geq 3$;} \\
            x+3y, & \hbox{ if $s=2$.}
          \end{array}
        \right. $$
Also, $a^E_{s}=2(x+y),  \,\,c^E_s=x+y$,  $$d^E_{s}=\left\{
          \begin{array}{ll}
            2(x+y), & \hbox{ if $s\geq 3$;} \\
            x+3y, & \hbox{ if $s=2$}
          \end{array}
        \right.  \hbox{ and }  b^E_{s}=\left\{
          \begin{array}{ll}
            3(x+y), & \hbox{ if $s\geq 4$;} \\
            2x+4y, & \hbox{ if $s=3$;} \\
              x+5y, & \hbox{ if $s=2$.} \\
          \end{array}
        \right.$$
 The conclusion is immediate by Theorem \ref{mainscp} and the condition that $x\leq y$.
 \end{proof}
 By calculating on $L^N_s$ and $L^E_s$ and using the same arguments as above, we also have:
  \begin{cor} Let $s,x,y$ be positive integers with $x\leq y$ and $s\geq 2$.  Let $$\a=((sx+(s-1)y)^x, (sx+(s-2)y)^x,\dots, (sx)^x,((s-1)x)^y,((s-1)x)^y,\dots,x^y),$$
 and $\beta$ be partitions of $n=sx(s(x+y)-y)$.  If $\beta$ contains a part $x+y$ or $2(x+y)$, then $\chi^\a(\beta)=0$.
 \end{cor}

\section*{Acknowledgements}
The authors would like to thank anonymous referee for improving the manuscript and would like also to thank Ratsiri Sanguanwong for his useful conversation.

\end{document}